\documentclass[12pt]{amsart}
\usepackage[dvips]{graphicx}
\usepackage{amssymb,amsmath,amsthm}
\usepackage{amsthm}
\usepackage{amssymb}
\usepackage{amsmath}
\usepackage{color}
\usepackage{textcomp}
\usepackage{amsthm,bm,comment}

\def\NZQ{\mathbb}               

\def\QQ{{\NZQ Q}}
\def\ZZ{{\NZQ Z}}
\def\RR{{\NZQ R}}

\def\frk{\mathfrak}               

\def\Phi{{\frk N}}

\def\ab{{\bold a}}

\def\xb{{\bold x}}

\def\opn#1#2{\def#1{\operatorname{#2}}} 
\opn\chara{char} \opn\length{\ell} \opn\pd{pd} \opn\rk{rk}
\opn\projdim{proj\,dim} \opn\injdim{inj\,dim} \opn\rank{rank}
\opn\depth{depth} \opn\grade{grade} \opn\height{height}
\opn\size{size}
\opn\embdim{emb\,dim} \opn\codim{codim}

\opn\Tr{Tr} \opn\bigrank{big\,rank}
\opn\superheight{superheight}\opn\lcm{lcm}
\opn\trdeg{tr\,deg}
\opn\reg{reg} \opn\lreg{lreg} \opn\ini{in} \opn\lpd{lpd}
\opn\size{size}\opn{\mult}{mult}
\opn{\Cl}{Cl}
\opn\div{div} \opn\Div{Div} \opn\cl{cl} \opn\Cl{Cl}
\opn\Spec{Spec} \opn\Supp{Supp} \opn\supp{supp} \opn\Sing{Sing}
\opn\Ass{Ass} \opn\Min{Min} \opn\cl{cl}
\opn\Ann{Ann} \opn\Rad{Rad} \opn\Soc{Soc}
\opn\Syz{Syz} \opn\Im{Im} \opn\Ker{Ker} \opn\Coker{Coker}
\opn\Am{Am} \opn\Hom{Hom} \opn\Tor{Tor} \opn\Ext{Ext}
\opn\End{End} \opn\Aut{Aut} \opn\id{id} \opn\ini{in}

\opn\nat{nat}
\opn\pff{pf}
\opn\Pf{Pf} \opn\GL{GL} \opn\SL{SL} \opn\mod{mod} \opn\ord{ord}
\opn\Gin{Gin}
\opn\Hilb{Hilb}\opn\adeg{adeg}\opn\std{std}\opn\ip{infpt}
\opn\Pol{Pol}
\opn\sat{sat}
\opn\Var{Var}
\opn\Gen{Gen}
\opn\lex{lex}
\opn\div{div}

\opn\aff{aff} \opn\con{conv} \opn\relint{relint} \opn\st{st}
\opn\lk{lk} \opn\cn{cn} \opn\core{core} \opn\vol{vol}
\opn\link{link} \opn\star{star}
\opn\gr{gr}

\def\Cc{{\mathcal C}}

\def\pot#1#2{#1[\kern-0.28ex[#2]\kern-0.28ex]}

\opn\dirlim{\underrightarrow{\lim}}
\opn\inivlim{\underleftarrow{\lim}}
\def\Implies{\ifmmode\Longrightarrow \else
        \unskip${}\Longrightarrow{}$\ignorespaces\fi}
\def\implies{\ifmmode\Rightarrow \else
        \unskip${}\Rightarrow{}$\ignorespaces\fi}
\def\iff{\ifmmode\Longleftrightarrow \else
        \unskip${}\Longleftrightarrow{}$\ignorespaces\fi}

\let\:=\colon

\def\NZQ{\mathbb}        

\def\QQ{{\NZQ Q}}
\def\ZZ{{\NZQ Z}}
\def\RR{{\NZQ R}}

\def\frk{\mathfrak}        

\def\Phi{{\frk N}}

\def\ab{{\bf a}}

\def\eb{{\bf e}}

\def\vb{{\bf v}}
\def\xb{{\bf x}}

\def \P{{\mathcal P}}

\newtheorem{thm}{Theorem}[section]
\newtheorem{lem}[thm]{Lemma}
\newtheorem{cor}[thm]{Corollary}
\newtheorem{prop}[thm]{Proposition}

\theoremstyle{definition}
\newtheorem{rem}[thm]{Remark}

\newtheorem{ex}[thm]{Example}

\newcommand{\mbe}{\operatorname{Ehr}}
\newcommand{\cone}{\operatorname{cone}}
\newcommand{\conv}{\operatorname{conv}}

\def\NZQ{\mathbb}    

\def\QQ{{\NZQ Q}}
\def\ZZ{{\NZQ Z}}
\def\RR{{\NZQ R}}

\def\frk{\mathfrak}    

\def\Phi{{\frk N}}

\def\ab{{\bf a}}

\def\eb{{\bf e}}
\def\nb{{\bf n}}

\def\vb{{\bf v}}
\def\xb{{\bf x}}

\def\qb{{\bf q}}
\def\wb{{\bf w}}

\def\P{{\mathcal P}}

\begin{document}

\title{Multibasic Ehrhart theory}
\author {Aki Mori, Takeshi Morita and Akihiro Shikama}

\thanks{}


\address{Aki Mori, Department of Pure and Applied Mathematics, Graduate School of Information Science and Technology,
Osaka University, Toyonaka, Osaka 560-0043, Japan}
\email{a-mori@cr.math.sci.osaka-u.ac.jp}
\address{Takeshi Morita, Department of Pure and Applied Mathematics, Graduate School of Information Science and Technology,
Osaka University, Toyonaka, Osaka 560-0043, Japan }
\email{t-morita@cr.math.sci.osaka-u.ac.jp}
\address{Akihiro Shikama, Department of Pure and Applied Mathematics, Graduate School of Information Science and Technology,
Osaka University, Toyonaka, Osaka 560-0043, Japan}
\email{a-shikama@cr.math.sci.osaka-u.ac.jp}

\maketitle
\begin{abstract}
In the present paper, we introduce a multibasic extension of the Ehrhart theory. We give a multibasic extension of Ehrhart polynomials and Ehrhart series.  We also show that an analogue of Ehrhart reciprocity holds for multibasic Ehrhart polynomials.  
\end{abstract}

\section*{Introduction}
Recently, some extensions of the Ehrhart theory have been studied, for example in \cite{Cha1,Stap1,Stap2}. In this paper, we introduce a ``multibasic'' extension of  the Ehrhart theory \cite{Cha1}. In the study of the special functions and the difference equations, the $q$-analogues of the special functions are well known. For example, the  (generalized) basic hypergeometric series, the $q$-Bessel functions, the $q$-Airy functions...(see \cite{GR} for more details).  These functions have a parameter $q$, which is called ``the base''. If we introduce other bases $q_1, q_2, \dots , q_N$ where $q_j\not=q_k$, we may consider the multibasic extension  \cite{GS} of the original $q$-analogues. A $q$-analogue of the Ehrhart theory is given by F.~Chapoton \cite{Cha1}. He gave the $q$-extension of the Ehrhart theory by a suitable linear form. 

\bigskip
In the study of combinatorics, the Ehrhart theory plays an important role when we study the relationship between the convex polytopes and the integer points. We review the Ehrhart theory which was introduced by Eng\`ene Ehrhart in 1960s \cite{Ehr}. Let $\mathcal{P}\subset\RR^N$ be an integral convex polytope of dimension $d$ and the set $\mathcal{P}^\circ$ its interior. If $n$ is a positive integer then we define 
\[L_{\mathcal{P}}(n)=\#(n\mathcal{P}\cap \ZZ^N).\]
In other words, $L_{\mathcal{P}}(n)$ is equal to the number of the integer points in $n\mathcal{P}$ where $n\mathcal{P}=\{n \mbox{\boldmath$\alpha$} | \mbox{\boldmath$\alpha$}\in\mathcal{P}\}$.
Ehrhart showed that the enumerative function $L_{\mathcal{P}}(n)$ is a polynomial in $n$ of degree $d$. He also gave $L_{\mathcal{P}}(0)=1$. 
The polynomial $L_{\mathcal{P}}(n)$ is called the {\em Ehrhart polynomial} after his works. If we consider the Ehrhart polynomial $L_{\mathcal{P}}(n)$ where $n\in\ZZ$, we obtain the relation
\[L_{\mathcal{P}}(-n)=(-1)^dL_{\mathcal{P}^\circ}(n), \]
which is called the {\em Ehrhart reciprocity}. The generating function of the Ehrhart polynomial $L_{\mathcal{P}}(n)$ is 
given by 
\[\mbe_{\mathcal{P}}(t) =1+\sum_{n=1}^\infty L_{\mathcal{P}}(n)t^n.\]
This generating function is also called the {\em Ehrhart series} of $\mathcal{P}$. 
It is known that the Ehrhart series have the following representation by rational function
\[\mbe _{\mathcal{P}}(t)=
\frac{\delta_{d}t^d+\delta_{d-1}t^{d-1}+\dots +\delta_0}{(1-t)^{d+1}}\]
where $\delta_k\in\ZZ$, provided that $0\le k\le d$. 
The vector $\delta (\mathcal{P})=(\delta_0,\delta_1,\dots ,\delta_d)\in\ZZ^{d+1}$ is called the {\em $\delta$-vector} of $\mathcal{P}$.  
The $\delta$-vectors satisfy the equations $\delta_0=1$ and $\delta_1=\#(\mathcal{P}\cap\ZZ^N)-(d+1)$.  
A lot of significant other properties of the $\delta$-vectors are studied, for example in \cite{Higa,Stap3}. 
We refer the reader to \cite{BeRo,H} for the detailed information about Ehrhart theory.

\bigskip

This paper is organized as follows. In Section \ref{sec1}, we introduce the concepts of the multibasic Ehrhart series and the multibasic $\delta$-vector, which are multibasic extensions of each one. 
We show that the multibasic Ehrhart series has representation as a rational function. Its numerator is a polynomial in $t$ whose coefficients are given by the elements of $\ZZ [q_1, q_1^{-1},\dots ,q_N, q_N^{-1}]$, and the degree of the polynomial is $(\textrm{the number of vertices of}\,\,  \mathcal{P}) -1$, at most. In Section \ref{sec2},  we introduce the multibasic Ehrhart polynomials for integral polytopes $\mathcal{P} \subset (\RR_{\geq 0})^N$. 
We also give some examples of the multibasic Ehrhart polynomials for typical classes of polytopes. 
In Section \ref{sec3}, we obtain the reciprocity theorem for the multibasic Ehrhart polynomials.

\section{The multibasic Ehrhart series}\label{sec1}
In the present paper we assume that each polytope is convex.
Let $\ab =(a_1,\dots ,a_N)$
$\in\ZZ^N$ be an integer point. The Laurent monomial $\qb^\ab$  is defined as
\[\qb^\ab :=q_1^{a_1}q_2^{a_2}\cdots q_N^{a_N}, \quad \qb^{\bf{0}}:=1\]
where  ${\bf{0}}:=(0,0,\dots ,0)$.
Recall that for a given rational cone or rational polytope $S\subset \RR^N$, 
\[\sigma_S(\qb )=\sigma_S(q_1,q_2,\dots ,q_N):=\sum_{\ab\in S\cap \ZZ^N}\qb^\ab \]
is called the {\em integer-point transform} of $S$ \cite{BeRo}.  
Let $\mathcal{P}\subset \RR^N$ be an integral polytope of dimension $d$. For any $n\in\ZZ_{>0}$, let $n\mathcal{P}=\{n \mbox{\boldmath$\alpha$} | \mbox{\boldmath$\alpha$}\in\mathcal{P}\}$. Now we define the {\em multibasic Ehrhart series} as foliows:

\[\mbe_{\mathcal{P}}(t;\qb )=
\mbe_{\mathcal{P}}(t;q_1,q_2,\dots ,q_N)
:=1+\sum_{n=1}^\infty \sigma_{n\mathcal{P}}(\qb )t^n . \]
Note that 
this definition gives us the expansion of the classical Ehrhart series. Indeed the special case $\mbe_{\mathcal{P}}(t;1,1,\dots ,1)$ is the classical Ehrhart series. 


\begin{lem}[{\cite[Theorem3.5]{BeRo}}]
\label{mbeck}
For any $\wb_1, \wb_2,\dots ,\wb_d\in\ZZ^N$ 
if $\mathcal{K}
:=\{ \sum_{i=1}^dr_i\wb_i \mid r_i\ge 0\}$ 
is a simplicial $d$-cone,
then for $\vb\in\RR^N$, the integer-point transform of $\vb +\mathcal{K}$ is given by
\[
\sigma_{\vb +\mathcal{K}}(\qb )
=
\frac{\sigma_{\vb +\Gamma}(\qb)}{\displaystyle \prod_{i=1}^d(1-\qb^{\wb_i})}
,\]
where
$\Gamma :=\{\sum_{i=1}^d r_i\wb_i | 0 \le r_i < 1\}$. 
\end{lem}

\begin{prop}\label{mbe:simplex}
Let $\Delta\subset\RR^N$ be an integral simplex of dimension $d$ with vertices $\vb_1,\vb_2, \dots ,\vb_{d+1}$.  
Then the multibasic Ehrhart series of $\Delta$ is 
\[
\mbe_{\Delta} (t;\qb )=\frac{\displaystyle \sum_{\ab\in\Gamma\cap \ZZ^{N +1}}q_1^{a_1}
q_2^{a_2}\cdots q_N^{a_N}t^{a_{N+1}}}
{\displaystyle \prod _{i=1}^{d+1}(1-\qb^{\vb_i}t)},
\]
where $\Gamma =\{\sum_{i=1}^{d+1}r_i(\vb_i,1) |0\le r_i<1\}$. 
 More precisely the numerator of $\mbe_{\Delta} (t;\qb)$ is a polynomial in $t$, which has Laurent polynomial in $N$ variables as its coefficients and the coefficients of each Laurent polynomials is either $0$ or $1$.
Moreover the degree of this polynomial is at most $d$.
\end{prop}


\begin{proof}We define a set $\cone (\Delta)$ as follows:
\[\cone (\Delta ):=\left\{\left. \sum_{i=1}^{d+1}r_i(\vb_i,1)\right| r_i\ge 0 \right\}\subset \RR^{N+1}.\]
Note that $\cone (\Delta )$ is a simplicial $(d+1)$-cone, and we obtain $n\Delta$ by considering intersection of $\cone (\Delta )$ and a hyperplane $x_{N+1} = n$.  
Then we can calculate as follows by Lemma \ref{mbeck} :
\begin{align*}
\mbe_{\Delta}(t;\qb )&=\mbe_{\Delta}(t;q_1,q_2,\dots ,q_N)
=1+\sum_{n=1}^\infty \sigma_{n\Delta}(q_1,q_2,\dots ,q_N)t^n\\
&=\sum_{(a_1,\dots ,a_{N+1})\in\cone (\Delta )\cap\ZZ^{N+1}}
q_1^{a_1}q_2^{a_2}\cdots q_N^{a_N}t^{a_{N+1}}=\sigma_{\cone (\Delta )}(q_1,q_2,\dots ,q_N,t)\\
&=\frac{\sigma_\Gamma (q_1,q_2,\dots ,q_N,t)}{\prod_{i=1}^{d+1}(1-\qb^{\vb_i}t)}
=\frac{\sum_{\ab\in\Gamma\cap\ZZ^{N+1}}q_1^{a_1}q_2^{a_2}\cdots q_N^{a_N}t^{a_{N+1}}}{\prod_{i=1}^{d+1}(1-\qb^{\vb_i}t)}. 
\end{align*}
Since the $(N+1)$-th coordinate of each generator of $\cone(\Delta)$ is $1$, the $(N+1)$-th coordinate $a_{N+1}$ of $\ab \in \Gamma \cap \ZZ^{N+1}$ is $r_1+\cdots +r_{d+1}$, where $0 \le r_1 ,\ldots,r_{d+1} <1$. Hence since 
$r_1+\cdots +r_{d+1}$ is an integer, the degree in $t$ of numerator is at most $d$, as desired.
\end{proof}

\begin{ex}Let $\Delta =[a, b]$ be a $1$-simplex where $a,b\in\ZZ$ and $a<b$. Then we obtain 
\[\mbe_{\Delta}(t;\qb )=
\frac{1+\sum_{k=a+1}^{b-1}q_1^kt}{(1-q_1^at)(1-q_1^bt)}, 
\]
since $\Gamma \cap \ZZ^2=\{(0,0), (a+1,1), (a+2,1), \dots ,(b-1,1)\}$. 
\end{ex}
\begin{ex}[The standard $d$-simplex] Let $\eb_i$  $(1\le i\le d+1)$ be unit vectors in $\RR^{d+1}$ and let 
$\Delta :=\conv (\{\eb_1, \eb_2, \dots ,\eb_{d+1}\})$.
Then the multibasic Ehrhart series of $\Delta$ is given as 
\[\mbe_{\Delta}(t;\qb )=\frac{1}{\prod_{i=1}^{d+1}(1-q_it)}, \]
since we have $\Gamma \cap \ZZ^{d+2}=\{\bf 0\}$ by 
\[
\Gamma =\left\{\sum_{i=1}^{d+1}r_i(\eb_i,1) \biggl| 0\le r_i<1\right\}=\left\{(r_1,r_2,\dots ,r_{d+1},\sum_{i=1}^{d+1}r_i) \biggl| 0 \le r_i < 1\right\}.
\]
\end{ex}

\begin{rem}\label{re2}
 \[\sigma_{\cone (\mathcal{P})}(q_1,q_2,\dots ,q_N, q_{N+1})\]
is the integer-point transform of the cone over the integral polytope $\mathcal{P} \subset \RR^N$.
Then we have 
\[\sigma_{\cone (\mathcal{P})}(q_1,q_2,\dots ,q_N,t)=\mbe_{\mathcal{P}}(t;\qb )\]
and
\[\sigma_{\cone (\mathcal{P})}(1, 1, \dots ,1,t)=\mbe_{\mathcal{P}}(t)\]
by the same idea of  the proof of Proposition~\ref{mbe:simplex}.
Moreover we have 
\[
\sigma_{\cone (\mathcal{P})}(q^{a_1}, q^{a_2}, \dots ,q^{a_N},t)
=\mbe_{\mathcal{P},\lambda}(t,q) 
\]
where $\mbe_{\mathcal{P},\lambda}(t,q)$
is the $q$-Ehrhart series and $\lambda = (a_1,a_2, \ldots , a_N)$
is a linear form satisfying positivity and genericity (See \cite{Cha1}).

Namely, the classical Ehrhart series, the $q$-Ehrhart series and the multibasic Ehrhart series can be obtained from 
$\sigma_{\cone (\mathcal{P})}(q_1,q_2,\dots ,q_N, q_{N+1})$.
 \end{rem}

\begin{thm}\label{c1.9}Let $\P \subset \RR^N$ 
be an integral polytope of dimension $d$ and let $\vb_1,\ldots,\vb_m$ its vertices.
Then the multibasic Ehrhart series of $\mathcal{P}$ can be displayed as 
\begin{align*}
\mbe_{\mathcal{P}}(t;\qb )
=
\frac{\delta_{m-1} t^{m-1}+\delta _{m-2}t^{m-2}+ \cdots +\delta_1 t +\delta_0}
{\displaystyle \prod_{i=1}^m(1-\qb^{\vb_i}t)},\\
\delta_k \in \ZZ [\qb,\qb^{-1}] = \ZZ[q_1,q_1^{-1},\ldots 
,q_N,q_N^{-1}], \; \; 0\le k \le m-1 
.
\end{align*}
Namely, the numerator of the multibasic Ehrhart series is a polynomial in $t$, which has 
Laurent polynomials in $N$ variables with coefficients in $\ZZ$
as its coefficients.
In particular the degree of the numerator of $\mbe_{\mathcal{P}}(t;\qb )$
is at most $m-1$.
\end{thm}
\begin{proof}
Since the set $\cone (\mathcal{P})$ is a pointed cone, $\cone (\mathcal{P})$ can be triangulated into simplicial cones such that each simplicial cone does not have any  new generator.
We remark that the intersection of simplicial cones is a  simplicial cone. 
We can obtain the integer-point transform of each simplicial cone by Lemma \ref{mbeck}. 
By the principle of inclusion-exclusion, we obtain the integer-point transform of $\cone (\P )$ as an alternating sum of the integer-point transforms of those simplicial cones.  
Thus the integer-point transform of $\cone (\mathcal{P})$ is given by
\[
\sigma_{\cone (\mathcal{P})}(q_1, q_2,\dots ,q_{N+1})
=
\frac{f}{\prod_{i=1}^m (1-\qb^{\vb_i}q_{N+1})}
 \]
where $f\in\ZZ[q_1,q_1^{-1},\ldots ,q_{N+1},q_{N+1}^{-1}]$.
By Proposition~\ref{mbe:simplex}, the highest exponent of $q_{N+1}$ that appear in the numerator of 
the integer-point transform of each $(d+1)$-dimensional simplicial cone  is  at most $d$.
Thus the highest exponent of $q_{N+1}$ that appear in $f$ is at most $d+ \{m-(d+1)\}=d-1+\{m-(d-1+1)\}= m-1$. 
Since $\sigma_{\cone (\mathcal{P})}(q_1,q_2,\dots ,q_{N},t)=\mbe_{\mathcal{P}}(t;\qb )$, 
 we obtain the conclusion.  
\end{proof}

We call 
\[
\delta_{{\bf q}}({\mathcal{P}}):=(\delta_0,\delta_1,\ldots,\delta_{m-1})
,\;\;\;
\delta_{k} \in \ZZ[{\bf q},{\bf q}^{-1}] 
\]
the {\em multibasic $\delta$-vector} of $\P$, where 
$\sum_{k=0}^{m-1} \delta_k t^k$
is the numerator of 
${\rm Ehr}_{\mathcal{P}}(t; {\bf q})$
that we acquire in Theorem~\ref{c1.9}. 

\begin{cor}
Let $\P \subset \RR^N $ be an integral polytope and 
${\bf v}_1,{\bf v}_2,\ldots,{\bf v}_m$ its vertices.
Then for the multibasic $\delta$-vector of $\P$, the following properties hold.
\[
\delta_0 = 1,\;\;\; \delta_1=\sigma_{\mathcal{P}}({\bf q}) - \sum_{i=1}^m {\bf q}^{{\bf v}_i}
\]
\end{cor}


\begin{cor}
Let $\mathcal{P} \subset \RR^N$ be an integral polytope and
let $\P' = \P + \vb , \vb \in \ZZ ^N$.
Then we have
\[
{\rm Ehr}_{\mathcal{P}'}(t;{\bf q})
=
{\rm Ehr}_{\mathcal{P}}({\bf q}^{\bf v}t;{\bf q}),
\]
\[
\delta_{{\bf q}}({\mathcal{P}'})
=(\delta_0,\delta_1{\bf q}^{\bf v},\delta_2{\bf q}^{2{\bf v}},
\ldots,\delta_{m-1}{\bf q}^{(m-1){\bf v}}),\;\;\;
\delta_{k} \in \ZZ[{\bf q},{\bf q}^{-1}].
\]

\end{cor}

\begin{proof}
Since we have 
$
\sigma_{n{\mathcal{P}'}}({\bf q})
=
\sigma_{n{\mathcal{P}}}({\bf q}) \cdot {\bf q}^{n {\bf v}}
$
for all $n \in \ZZ_{>0}$, it then follows that
\begin{align*}
{\rm Ehr}_{\mathcal{P}'}(t;{\bf q})
&=
1+\sum_{n=1}^{\infty}
\sigma_{n{\mathcal{P}'}}({\bf q})t^n
=
1+\sum_{n=1}^{\infty}
\sigma_{n{\mathcal{P}}}({\bf q}) \cdot {\bf q}^{n {\bf v}}t^n
\\
&=
1+\sum_{n=1}^{\infty}
\sigma_{n{\mathcal{P}}}({\bf q})({\bf q}^{\bf v}t)^n
=
{\rm Ehr}_{\mathcal{P}}({\bf q}^{\bf v}t;{\bf q}).
\end{align*}
Hence the numerator of  ${\rm Ehr}_{\mathcal{P}'}(t;{\bf q})$ is
$
\sum_{k=0}^{m-1} \delta_k ({\bf q}^{\bf v}t)^k
=
\sum_{k=0}^{m-1} (\delta_k {\bf q}^{k{\bf v}}) t^k
$, as desired.
\end{proof}

\section{Multibasic Ehrhart polynomials}\label{sec2}
In this section, we show the existence of a multibasic Ehrhart polynomial. The notation
\[[n]_q:=\frac{1-q^n}{1-q}, \quad n\in\ZZ\]
is a $q$-integer \cite{GR}. We review the key relation as follows:
\begin{equation*}
\left.(1+qx-x)\right|_{x=[n]_q}
=1+q\frac{1-q^n}{1-q}-\frac{1-q^n}{1-q}=q^n.
\end{equation*}
Then we have the following relation immediately:
\begin{equation}\label{qint}  
\qb^{\nb}=\left.\prod_{k=1}^N(1+q_kx_k-x_k)\right|_{x_k=[n]_{q_k}}
\end{equation} 
where $\nb = (n,n,\dots ,n)$. 
\bigskip

For any vertex $\vb$ of a rational polytope $\mathcal{P} \subset \RR^N$, the {\em vertex cone} $\mathcal{K}_\vb$ is defined by 
\[\mathcal{K}_\vb :=\{\vb +r(\xb -\vb ) \mid \xb\in\mathcal{P}, r\in\RR_{\ge 0}\}. \]
Let $\vb_i$ be a vertex of $\P$ and set 
\[\mathcal{C}_i:=(-\vb_i)+\mathcal{K}_{\vb_i}=\{r(\xb -\vb_i ) \mid \xb \in\P , r\in \RR_{\ge 0}\}.\]

\begin{lem}[Brion's theorem,\cite{Br}]\label{brion}
Let $\mathcal{P}\subset \RR^N$ be a rational polytope of dimension $d$. Then
\[\sigma_{\mathcal{P}}(\qb )
=
\sum_{\vb \textrm{ a vertex of }\mathcal{P}}\sigma_{\mathcal{K}_\vb}(\qb ).\]
\end{lem}

\begin{thm}\label{mori1}
Let $\mathcal{P}\subset (\RR_{\ge 0})^N$ be an integral polytope of dimension $d$.  Then there exists a polynomial 
\[L_{\mathcal{P}}(x_1,x_2,\dots ,x_N)
\in
\mathbb{Q}(\qb )[x_1,x_2,\dots ,x_N]\]
such that 
\[L_{\mathcal{P}}([n]_{q_1}, [n]_{q_2}, \dots ,[n]_{q_N})
=\sigma _{n\mathcal{P}}(\qb ),\quad  \forall n\in \ZZ_{> 0}.\] 
In particular, the degree of $L_{\P}(x_1,x_2,\dots ,x_N)$ coincides with 
$\max \{\sum_{k=1}^Nv_{ik}|1\le i\le m\}$ where $\vb_i=(v_{i1},v_{i2},\dots ,v_{iN}),1\le i\le m $ are vertices of $\P$.  
\end{thm}


\begin{proof} By Lemma \ref{brion}, we have 
\[
\sigma_{n{\mathcal{P}}}({\bf q})
=
\sum_{i=1}^m \sigma_{\mathcal{K}_{n{\bf v}_i}}({\bf q}),
\]
since $\mathcal{K}_{n{\bf v}_i}$ can be obtained as a parallel
translation of $\Cc_i$, then we have 
$
\sigma_{\mathcal{K}_{n{\bf v}_i}}({\bf q})
=
\sigma_{\mathcal{C}_i}({\bf q})\cdot{\bf q}^{{n{\bf v}_i}}
$.  
 Therefore we obtain
\[
\sum_{i=1}^m \sigma_{\mathcal{K}_{n{\bf v}_i}}({\bf q})
=
\sum_{i=1}^m \sigma_{\mathcal{C}_i}({\bf q})\cdot{\bf q}^{{n{\bf v}_i}}.
\]
By using \eqref{qint}, we have 
\[
{\qb}^{n\vb_i}  =  \prod_{k = 1} ^ N(1+q_kx_k-x_k)^{v_{ik}} \bigg| _{x_k = [n]_{q_k}}
\]
Thus, we have
\[
L_{\mathcal{P}}(x_1,x_2,\ldots,x_N)
=
\sum_{i=1}^m \sigma_{\mathcal{C}_i}({\bf q})
\prod_{k=1}^N (1+q_kx_k-x_k)^{{v_{ik}}}
\]

Our claim on the degree of $L_P(x_1,x_2,\ldots, x_N)$ follows immediately
since \\$\sigma_{c_i}(\qb) \in \QQ(\qb)$.

\end{proof}
We call the polynomial $L_{\mathcal{P}}(x_1,x_2,\dots ,x_N)$ the {\em multibasic Ehrhart polynomial} of $\mathcal{P}$. 

\begin{cor}Let $\P\subset(\RR_{\ge 0})^N$ be a integral polytope, $\P '=\P +\vb\subset(\RR_{\ge 0})^N$ and $\vb =(v_1,v_2,\dots ,v_N)\in\ZZ^N$. Then we have 
\[L_{\P '}(x_1,x_2,\dots x_N)
=
L_{\P}(x_1,x_2,\dots ,x_N)\prod_{k=1}^N(1+q_kx_k-x_k)^{v_k}.  \]

\end{cor}

\begin{proof} By Theorem \ref{mori1}, we have 
\begin{align*}L_{\P '} (x_1,x_2,\dots ,x_N)&=\sum_{i=1}^m\sigma_{\mathcal{C}_i}(\qb )\prod _{k=1}^N(1+q_kx_k-x_k)^{v_{ik}+v_k}\\
&=L_{\P}(x_1,x_2,\dots ,x_N)\prod_{k=1}^N(1+q_kx_k-x_k)^{v_k}.
\end{align*}
\end{proof}

\begin{lem}\label{mori2}The sum of each integer-point transform of $\mathcal{C}_i$ is equals to one, namely, 
\[
\sum_{i=1}^m\sigma_{\mathcal{C}_i}(\qb )=1.
\]
\end{lem}

\begin{proof}
Let $\mathcal{Q}$ be the image of $\mathcal{P}$ by a dilation and translation
and  ${\bm \varepsilon}_1,\ldots, {\bm \varepsilon}_m \in \RR^N$ its vertices, that is
\[
\mathcal{Q}:=r\mathcal{P}+{\bm \alpha}=\textrm{conv}(\{{\bm \varepsilon}_1,\ldots, {\bm \varepsilon}_m\})
,\;\;
r \in \RR, {\bm \alpha} \in \RR^N. 
\]
Then we can choose ${\bm \varepsilon}_1,\ldots, {\bm \varepsilon}_m$ 
in such a way that:
\begin{itemize}
\item 
the origin of $\RR^N$ is a unique integer point in $\mathcal{Q}$,
\[
\mathcal{Q} \cap \ZZ^N = \{{\bf 0}\};
\]
\item
the integer points contained in each vertex cone of $\mathcal{Q}$
are precisely the integer points contained in the corresponding vertex cone $\mathcal{C}_i$,
\[
\mathcal{K}_{{\bm \varepsilon}_i} \cap \ZZ^N
=
\mathcal{C}_i \cap \ZZ^N
,\;\;\;\;
\textrm{for all}\;\; i = 1,\ldots,m.
\]
\end{itemize}
By the two conditions and Lemma \ref{brion}, we obtain 
\[
1
=
\sigma_{\mathcal{Q}}(\qb )
=
\sum_{\vb \textrm{ a vertex of } \mathcal{Q}}\sigma_{\mathcal{K}_\vb}(\qb )
=
\sum_{i=1}^m\sigma_{\mathcal{K}_{{\bm \varepsilon}_i}}(\qb )
=
\sum_{i=1}^m\sigma_{\mathcal{C}_i}(\qb ).
\]

\end{proof}

\begin{cor}Let $\P\subset\RR^N$ be an integral polytope. Then the constant part of the multibasic Ehrhart polynomial $L_\P(x_1,x_2,\dots ,x_N)$ is equals to one.  
\end{cor}

\begin{proof}
By substituting $x_i=0$ to $L_{\P}(x_1,x_2,\dots ,x_N)$, we have
\[L_{\P}(0,0,\dots ,0)
=\sum_{i=1}^m\sigma_{\mathcal{C}_i}(\qb ). \]
By Lemma \ref{mori2}, we obtain the conclusion. 

\end{proof}

Next, we give examples of the multibasic Ehrhart polynomials by Lemma \ref{mori2}. We can compute  $L_{\mathcal{P}}(x_1,x_2,\dots ,x_N)$ for the classes of polytopes.

\begin{ex}\label{mp2.7}
We fix the sets as follow:
\begin{enumerate}
\item 
The $d$-simplex of $\RR^{d+1}$ is given by 
$\Delta :=\conv (\{\eb_1, \eb_2, \dots ,\eb_{d+1}\})$,
\item 
The $d$-simplex of $\RR^{d}$ is given by 
$\Delta ':=\conv (\{{\bf 0}, \eb_1, \eb_2, \dots ,\eb_{d}\})$, 
\item 
The unit $d$-cube is given by 
$\square :=\conv(\{(x_1,x_2,\dots ,x_d) \mid x_i =0\,\textrm{or}\,1, 1\le i\le d\})$.
\end{enumerate}
Then each multibasic Ehrhart polynomial is given by: 
\begin{enumerate}
\item $L_{\Delta}(x_1,x_2,\dots ,x_{d+1})=\displaystyle\sum_{i=1}^{d+1}\dfrac{q_i^d(q_i-1)}{ 
\displaystyle \prod_{\stackrel{1\leq j \leq d+1}{j \neq i}}(q_i-q_j)}x_i+1$, 
\item $L_{\Delta '}(x_1,x_2,\dots ,x_d)=\displaystyle \sum _{i=1}^d \dfrac{q_i^d}{ 
\displaystyle\prod_{\stackrel{1\leq j \leq d+1}{j \neq i}}(q_i-q_j)}x_i+1 $,
\item $L_\square (x_1,x_2,\dots ,x_d)=\displaystyle\prod_{1\le i\le d}(q_ix_i+1)$.
\end{enumerate}
\end{ex}

\begin{proof}We consider each case. 
\begin{enumerate}\item The $d$-simplex $\Delta \subset \RR^{d+1}$ case. 

For any $i$\,(provided that $1\le i\le d+1$)
, the generator of each vertex cone $\mathcal{K}_{{\eb}_i}$ is given by $\{\eb_j-\eb_i\mid1\le j\le d+1, j\not= i\}$. Then we have 
\begin{align*}
\Gamma_i
:=
&\left\{\left.\sum_{\stackrel{1\le j\le d+1}{j\not= i}}r_j(\eb_j-\eb_i)\right|0\le r_j< 1\right\}\\
=& \left\{\left.\left(r_1,r_2,\dots ,r_{i-1}, -\sum_{\stackrel{1\le j\le d+1}{j\not=i}}r_j, r_{i+1},\dots , r_{d+1}\right)
\right|0\le r_j< 1 
\right\}.
\end{align*}
Namely, we have $\Gamma_i\cap \ZZ^{d+1}=\{\bf{0}\}$ for any $i$. By Lemma \ref{mbeck}, 
\[\sigma_{\mathcal{C}_i}(\qb ) =
\sigma_{(-\eb_i)+\mathcal{K}_{\eb_i}}(\qb )
=\frac{1}{\prod_{\stackrel{1\le j\le d+1}{j\not= i}}(1-q_i^{-1}q_j)}.\]
Combining Lemma \ref{mori2} and Theorem \ref{mori1}, we obtain
\begin{align*}
L_{\Delta}(x_1,x_2,\dots , x_{d+1})&=\sum_{i=1}^{d+1}\sigma_{\mathcal{C}_i}(\qb )(1+q_ix_i-x_i)\\
&=\sum_{i=1}^{d+1}\sigma_{\mathcal{C}_i}(\qb )(q_i-1)x_i+\sum_{i=1}^{d+1}\sigma_{\mathcal{C}_i}(\qb )\\
&=\sum_{i=1}^{d+1}\frac{q_i-1}{\prod_{\stackrel{1\le j\le d+1}{j\not= i}}(1-q_i^{-1}q_j)}x_i+1\\
&=\sum_{i=1}^{d+1}\frac{q_i^d(q_i-1)}{\prod_{\stackrel{1\le j\le d+1}{j\not= i}}(q_i-q_j)}x_i+1. 
\end{align*}

\item The $d$-simplex $\Delta '\subset \RR^d$ case. 

We remark that the generator of the vertex cone $\mathcal{K}_{\bf 0}$ is $\{\eb_1,\eb_2,\dots ,\eb_d\}$ and the generator of each vertex cone $\mathcal{K}_{\eb_i}$ is given by $\{{\bf 0}-\eb_i\} \cup \{\eb_j-\eb_i \mid1\le j\le d+1, j \not= i\}$. Then we have
\begin{align*}
\Gamma_0
&:=
\left\{\left.\sum_{j=1}^{d} r_j {\bf e}_j
\,\right|\,
0 \leq r_j < 1\right\}
=
\left\{\left.(r_1,r_2, \ldots, r_d)
\,\right|\,
0 \leq r_j < 1\right\},
\\
\Gamma_i
&:=
\left\{\left.
r_0({\bf 0}-{\bf e}_i) 
+
\sum_{\stackrel{1\leq j \leq d}{j \neq i}} 
r_j ({\bf e}_j-{\bf e}_i)
\,\right|\,
0 \leq r_0, r_j < 1\right\}
\\
&=
\left\{\left.
(r_1,r_2, \ldots, r_{i-1}, -{\sum_{\stackrel{0 \leq j \leq d}{j \neq i}}r_j}, 
r_{i+1}, \ldots, r_d) 
\,\right|\,
0 \leq r_0, r_j < 1
\right\}.
\end{align*}
Therefore, we have $\Gamma_0\cap\ZZ^d=\Gamma_i\cap\ZZ^d=\{{\bf 0}\}$, provided that $1\le i\le d$. 
By Lemma \ref{mbeck}, we  obtain 
\begin{align*}
\sigma_{\mathcal{C}_0} ({\bf q})
&=
\sigma_{\mathcal{K}_{\bf 0}} ({\bf q})
= 
\frac{1}
{\displaystyle \prod_{1\leq j \leq d} (1-q_j)},
\\
\sigma_{\mathcal{C}_i} ({\bf q})
&=
\sigma_{(-{{\bf e}_i})+\mathcal{K}_{{\bf e}_i}} ({\bf q})
= 
\frac{1}
{(1-q_i^{-1})\displaystyle \prod_{\stackrel{1\leq j \leq d}{j \neq i}}(1-q_i^{-1}q_j)}.
\end{align*}
By Lemma \ref{mori2} and Theorem \ref{mori1}, we have
\begin{align*}
L_{\Delta'}(x_1,x_2,\ldots,x_d)
&=
\sigma_{\mathcal{C}_0} ({\bf q})
+
\sum_{i=1}^d
\sigma_{\mathcal{C}_i} ({\bf q})
(1+q_ix_i-x_i)
\\
&=
\sum_{i=1}^{d}
\sigma_{\mathcal{C}_i} ({\bf q}) (q_i-1)x_i
+
\sum_{i=0}^{d}
\sigma_{\mathcal{C}_i} ({\bf q})
\\
&=
\sum_{i=1}^{d}
\frac{q_i-1}
{(1 - q_i^{-1})\displaystyle \prod_{\stackrel{1 \leq j \leq d}{j \neq i}}(1 - q_i^{-1}q_j)}
x_i
+
1\\
&=
\sum_{i=1}^d
\frac{q_i^d}
{\displaystyle \prod_{\stackrel{1 \leq j \leq d}{j \neq i}}(q_i - q_j)}
x_i
+
1.
\end{align*}
\item The unit $d$-cube $\square\subset \RR^d$ case. 

If we consider the special case $x_i=[n]_{q_i}$, we obtain 
\begin{align*}
L_{\Box}([n]_{q_1},[n]_{q_2},\ldots,[n]_{q_d})&=
\prod_{1 \leq i \leq d}(1+q_i+q_i^2+\cdots+q_i^n)\\
&=
\sum_{{\bf a} \in n\Box \cap \ZZ^d} {\bf q}^{\bf a}
=
\sigma_{n \Box}({\bf q}).
\end{align*}
\end{enumerate}
\end{proof}

\section{Multibasic Ehrhart reciprocity}\label{sec3}
Let $\P^\circ $ be the interior of $\P$. 
We define the multibasic Ehrhart series for the interior of the polytope $\mathcal{P}\subset \RR^N$ of dimension $d$ as follows:
\[\mbe_{\mathcal{P}^\circ}(t;\qb ):=\sum_{n=1}^\infty \sigma_{n\mathcal{P}^\circ}(\qb )t^n. \]
We consider a multibasic analogue of the Ehrhart reciprocity. 

\begin{lem}[Stanley's reciprocity theorem, \cite{Stan} ]\label{Stanrey}
Let $\mathcal{K}\subset \RR^N$ be a rational $d$-cone with the origin as apex. Then  
\[
\sigma_{\mathcal{K}}\left(\frac{1}{q_1},\frac{1}{q_2},\dots , \frac{1}{q_N}\right)
=
(-1)^d\sigma_{\mathcal{K}^\circ}(q_1,q_2,\dots ,q_N). 
\]
\end{lem}

We also prepare the following lemma to study the reciprocity. 
\begin{lem}\label{st2}Let $\mathcal{P}\subset \RR^N$ be an integral polytope . For any $n\in\ZZ_{>0}$, we have 
\[\sum_{n\le 0}L_{\mathcal{P}}([n]_{q_1}, [n]_{q_2}, \dots ,[n]_{q_N})t^n
+
\sum_{n\ge 1}
L_{\mathcal{P}}([n]_{q_1}, [n]_{q_2},\dots , [n]_{q_N})t^n
=0.
\]
\end{lem}
\begin{proof}
Let $\vb_1, \vb_2,\dots ,\vb_m$ be the vertices of $\mathcal{P}$. We recall from Theorem \ref{mori1} that  there exists a relation 
\[L_{\mathcal{P}}([n]_{q_1}, [n]_{q_2}, \dots ,[n]_{q_N})
=
\sum_{i=1}^m \sigma_{\mathcal{C}_i}(\qb )\cdot \qb^{n\vb_i}. \]
Then we have 
\begin{align*}
&\sum_{n \leq 0}
L_{\mathcal{P}}
([n]_{q_1},[n]_{q_2},\ldots,[n]_{q_N})
t^n
+
\sum_{n \geq 1}
L_{\mathcal{P}}
([n]_{q_1},[n]_{q_2},\ldots,[n]_{q_N})
t^n
\\
&=
\sum_{n \leq 0}
\Bigr\{
\sum_{i=1}^m
\sigma_{\mathcal{C}_i}({\bf q})\cdot{\bf q}^{{n{\bf v}_i}}
\Bigl\}
t^n
+
\sum_{n \geq 1}
\Bigr
\{
\sum_{i=1}^m
\sigma_{\mathcal{C}_i}({\bf q})\cdot{\bf q}^{{n{\bf v}_i}}
\Bigl\}
t^n
\\
&=
\sum_{i=1}^m
\sigma_{\mathcal{C}_i}({\bf q})
\Bigr
\{
\sum_{n \geq 0}
({\bf q}^{-{{\bf v}_i}}t^{-1})^n
\Bigl\}
+
\sum_{i=1}^m
\sigma_{\mathcal{C}_i}({\bf q})
\Bigr
\{
\sum_{n \geq 1}
({\bf q}^{{\bf v}_i}t)^n
\Bigl\}
=
0.
\end{align*}
 Therefore, we obtain the conclusion. 
\end{proof} 

\begin{thm}
Let $\mathcal{P}\subset\RR^N$ be an integral polytope of dimension $d$. 
 For any $n\in\ZZ_{>0}$, we have 
\[L_{\mathcal{P}}([-n]_{q_1}, [-n]_{q_2}, \dots , [-n]_{q_N})
=
(-1)^d\sigma_{n\mathcal{P}^\circ}\left(\frac{1}{\qb }\right).\]
\end{thm}
\begin{proof}By applying Lemma \ref{Stanrey} to the cone over $\mathcal{P}$, 
\[\sigma_{\cone (\mathcal{P})}\left(\frac{1}{q_1}, \frac{1}{q_2}, \dots ,\frac{1}{q_N}, \frac{1}{q_{N+1}}\right)
=
(-1)^{d+1}\sigma_{(\cone (\mathcal{P}))^\circ}(q_1,q_2,\dots ,q_N,q_{N+1}). 
\]
We consider the special case $q_{N+1}=t$. By Remark \ref{re2},  
\[\mbe_{\mathcal{P}}\left(\frac{1}{t};\frac{1}{\qb}\right)
=
(-1)^{d+1}\mbe_{\mathcal{P}^\circ}(t;\qb ). 
\]
Then we have 
\[1+\sum_{n=1}^\infty\sigma_{n\mathcal{P}}\left(\frac{1}{\qb }\right)\left(\frac{1}{t}\right)^n
=(-1)^{d+1}\sum_{n=1}^\infty\sigma_{n\mathcal{P}^\circ}(\qb )t^n. 
\]
By Theorem \ref{mori1}, we have 
\[1+\sum_{n=1}^\infty\sigma_{n\mathcal{P}}\left(\frac{1}{\qb }\right)\left(\frac{1}{t}\right)^n
=
\sum_{n\le 0}L_{\mathcal{P}}([-n]_{1/q_1}, [-n]_{1/q_2}, \dots , [-n]_{1/q_N})t^n.
\]
We also obtain 
\[\sum_{n\le 0}L_{\mathcal{P}}([-n]_{1/q_1},\dots , [-n]_{1/q_N})t^n
=
-\sum_{n\ge 1}L_{\mathcal{P}}([-n]_{1/q_1}, \dots ,[-n]_{1/q_N})t^n. \]
by Lemma \ref{st2}. Therefore, 
\[\sum_{n\ge 1}L_{\mathcal{P}}([-n]_{1/q_1}, [-n]_{1/q_2},\dots , [-n]_{1/q_N})t^n
=
(-1)^d\sum_{n\ge 1}\sigma_{n\mathcal{P}^\circ}(\qb )t^n.
\]
 Finally, we acquire the conclusion. 
\end{proof}

\section*{Acknowledgement}
Our heartfelt appreciation go to Mr.~Akiyoshi Tsuchiya whose comments and suggestions were of inestimable value for our study.


\begin{thebibliography}{99}

\bibitem{BeRo}
M. Beck and S. Robins, ``Computing the continuous discretely,'' Undergraduate Texts in Mathematics, Springer, 2007.

\bibitem{Br}
M. Brion, Points entiers dans les poly\`edres convexes. {\it Ann. Sci. \'Ecole Norm. Sup.} 21.4 (1988), 653--663.

\bibitem{Cha1}
F. Chapoton, $q$-analogues of Ehrhart polynomials, 
arXiv:1301.1844v2.

\bibitem{Ehr}
E. Ehrhart, ``Polyn\^omes Arithm\'etiques et M\'ethode des Poly\`edres en Combinatoire,''
Birkh\"auser, Boston/Basel/Stuttgart, 1977.

\bibitem{GR}
G. Gasper and M. Rahman, ``Basic hypergeometric series,'' Vol. 96. Cambridge university press, 2004.

\bibitem{GS}
G.~Gasper and M.~Schlosser, Summation, transformation, and expansion formulas for multibasic theta hypergeometric series, \textit{Adv. Stud. Contemp. Math. (Kyungshang)} \textbf{11} (2005), 67--84. 

\bibitem{H}
T. Hibi, ``Algebraic Combinatorics on Convex Polytopes,'' Carslaw Publications, Glebe NSW, Australia,1992.

\bibitem{Higa}
A. Higashitani, Shifted symmetric $\delta$-vectors of convex polytopes, \textit{Discrete Math.} \textbf{310} (2010), 2925--2934.

\bibitem{Stan}
R. P. Stanley, Combinatorial reciprocity theorems,  {\it Adv. Math.}, \textbf{14} (1974), 194--253.

\bibitem{Stap1}
A. Stapledon, Weighted Ehrhart theory and orbifold cohomology, {\it Adv. Math.}, 219.1 (2008), 63--88.

\bibitem{Stap3}
A. Stapledon, Inequalities and Ehrhart $\delta$-vectors, \textit{Trans. Amer. Math. Soc.} \textbf{361} (2009), 5615--5626.

\bibitem{Stap2}
A. Stapledon, Equivariant Ehrhart theory,  {\it Adv. Math.}, 226.4 (2011), 3622--3654.


\end{thebibliography}
\end{document}